\documentclass[12pt,twoside, leqno]{amsart}

\usepackage{amsfonts, amssymb, amsmath, amsthm}

\allowdisplaybreaks[2]
\allowdisplaybreaks

\newcommand{\F}{\mathbb{F}}
\newcommand{\N}{\mathbb{N}}
\newcommand{\Z}{\mathbb{Z}}

\newcommand{\R}{\mathbb{R}}
\newcommand{\C}{\mathbb{C}}

\newcommand{\calF}{\mathcal{F}}
\newcommand{\calG}{\mathcal{G}}
\newcommand{\calI}{\mathcal{I}}

\newcommand{\stem}{\operatorname{stem}}
\newcommand\im{\operatorname{range}}
\newcommand\id{{\operatorname{id}}}
\newcommand{\lsp}{\operatorname{span}}
\newcommand{\clsp}{\overline{\lsp}}
\newcommand{\BS}{\operatorname{BS}}

\numberwithin{equation}{section}

\newtheorem{thm}{Theorem}[section]

\newtheorem{lemma}[thm]{Lemma}
\newtheorem{prop}[thm]{Proposition}
\theoremstyle{definition}

\newtheorem{defn}[thm]{Definition}
\newtheorem{example}[thm]{Example}
\newtheorem*{example*}{Example}

\begin{document}

\author[an Huef]{Astrid an Huef}
\author[Raeburn]{Iain Raeburn}
\author[Tolich]{Ilija Tolich}
\address{Department of Mathematics and Statistics, University of Otago, PO Box 56, Dunedin 9054, New Zealand.}
\email{\{astrid, iraeburn, itolich\}@maths.otago.ac.nz}


\thanks{This research was supported by the Marsden Fund of the Royal Society of New Zealand and by a University of Otago Publishing Bursary}

\title[HNN extensions of quasi-lattice ordered groups]{HNN extensions of quasi-lattice ordered groups and their operator algebras}

\begin{abstract} The Baumslag-Solitar group is an example of an HNN extension. 
Spielberg showed that it has a natural positive cone, and that it is then a quasi-lattice ordered group in the sense of Nica. We give conditions for an HNN extension of a quasi-lattice ordered group $(G,P)$ to be quasi-lattice ordered. In that case,  if $(G,P)$ is  amenable  as a quasi-lattice ordered group, then so is the HNN extension.
\end{abstract}

\date{8 March, 2017}
\maketitle

\section{Introduction}
Since they were introduced by Nica  \cite{Nica1992}, quasi-lattice ordered groups and their $C^*$-algebras have generated considerable interest (see, for example, \cite{LacaRaeburn1996},\cite{LacaRaeburn2010}). The amenability of quasi-lattice ordered groups has been a deep subject  (see, for example, \cite{CrispLaca2002},\cite{CrispLaca2007} and \cite{Li2012}). Quasi-lattice ordered groups are also   examples of the more recent LCM semigroups \cite{BrownloweLarsenStammeier2016}, \cite{Starling2015}.  
Here we generalise two recent results about the Baumslag-Solitar group.

First,  Spielberg proved   that the Baumslag-Solitar group is quasi-lattice ordered  \cite{Spielberg2012}. The Baumslag-Solitar group is an example of an HNN extension of $\Z$, and hence we wondered if HNN extensions could provide  new classes of quasi-lattice ordered groups. Spielberg also showed that a groupoid associated to  the Baumslag-Solitar semigroup is amenable \cite[Theorem~3.22]{Spielberg2012}.

Second, Clark, an Huef and Raeburn examined the phase-transitions of the Toeplitz algebra of the Baumslag-Solitar group \cite{HuefClarkRaeburn2015}. As part of their investigation they proved that the Baumslag-Solitar group is amenable as a quasi-lattice ordered group. The standard way to prove amenability, introduced by Laca and Raeburn \cite{LacaRaeburn1996}, is to use a ``controlled map'': an order-preserving homomorphism between quasi-lattice ordered groups. They  observed that the height map, which counts the number of times the stable letter of the HNN extension appears in a word, is almost a controlled map,  and then they adapted the standard proof in \cite[Appendix~A]{HuefClarkRaeburn2015}  to fit.

Our  innovation in this paper is a more general definition of a controlled map. We  prove in Theorem~\ref{p-controlled map implies amenable} that if $(G,P)$ is a quasi-lattice ordered group  and there is a controlled map $\mu$ into an amenable group,  and if $\ker\mu$  is an amenable quasi-lattice ordered group, then $G$  is amenable. The motivation for Theorem~\ref{p-controlled map implies amenable} was two-fold. First,  if a normal subgroup $N$ of a group $G$ is amenable and $G/N$ is amenable, then $G$ is amenable, and second,  Spielberg's  result  on  amenability of groupoids  \cite[Proposition~9.3]{Spielberg2014}.

In Theorem~\ref{(G*,P*) is left ql} we give conditions under which an HNN extension of a quasi-lattice ordered group is  quasi-lattice ordered. This result allows us to construct many new examples of quasi-lattice ordered groups. Finally, we use Theorem~\ref{p-controlled map implies amenable} to prove that an HNN extension of an amenable  quasi-lattice ordered group is  amenable (Theorem~\ref{HNN extensions are amenable as expected}).
\section{Preliminaries}
Let $P$ be a subsemigroup of a discrete group $G$ such that $P\cap P^{-1}=\{e\}$. There is a partial order on $G$ defined by
\[x\leq y\Leftrightarrow x^{-1}y\in P\Leftrightarrow y\in xP.
\]
 The order is left-invariant in the sense that $x\leq y$ implies $zx\leq zy$ for all $z \in G$.
A partially ordered group $(G, P)$ is \emph{quasi-lattice ordered} if every finite subset of $G$ with a common upper bound in $P$ has a least common upper bound in $P$  \cite[Definition 2.1]{Nica1992}.  By \cite[Lemma 7]{CrispLaca2002},   $(G,P)$ is  quasi-lattice ordered
if and only if:
\begin{equation}\label{Crisp Laca}
\ \hfill
\parbox{0.9\textwidth}{
 if $x\in PP^{-1}$, then there exist a pair  $\mu, \nu\in P$ with $x=\mu\nu^{-1}$  such that $\gamma, \delta\in P$ and $\gamma\delta^{-1}=\mu\nu^{-1}$ imply $\mu\leq \gamma$ and $\nu\leq \delta$. (The pair $\mu, \nu$ is unique.)}\end{equation}

Let $(G, P)$ be a quasi-lattice ordered group, and let $x, y\in G$.
If $x$ and $y$ have a common upper bound in $P$,  then their least common upper bound in $P$ is denoted   $x\vee y$. We write $x\vee y=\infty$ when $x$ and $y$ have no common upper bound in $P$ and $x\vee y<\infty$ when  they have  a common upper bound.
An \emph{isometric representation of $P$} in a $C^*$-algebra $A$ is a map $W:P\to A$ such that $W_e=1$, $W_p$ is an isometry and $W_pW_q=W_{pq}$  for all $p,q\in P$. We say that $W$ is \emph{covariant} if
\begin{equation}\label{covariance condition}
W_pW^*_pW_qW^*_q=\begin{cases}W_{p\vee q}W^*_{p\vee q}&\text{if $p\vee q<\infty$}\\0&\text{otherwise.}
\end{cases}
\end{equation}
Equivalently,  $W$ is covariant if
\[W^*_pW_q=\begin{cases}W_{p^{-1}(p\vee q)}W^*_{q^{-1}(p\vee q)}&\text{if $p\vee q<\infty$}\\0&\text{otherwise.}\end{cases}
\]
An example of a covariant representation is  $T:P\to B(\ell^2(P))$ characterised  by $T_p\epsilon_x=\epsilon_{px}$ where $\{\epsilon_x:x\in P\}$ is the orthonormal basis of point masses in $\ell^2(P)$.

In \cite[\S\S2.4 and 4.1]{Nica1992} Nica examined  two $C^*$-algebras associated to  $(G, P)$.
The reduced $C^*$-algebra $C^*_r(G,P)$ of $(G,P)$  is the $C^*$-subalgebra of $B(\ell^2(P))$ generated by $\{T_p:p\in P\}$.
The universal $C^*$-algebra $C^*(G,P)$  of $(G,P)$ is   generated by  a universal covariant representation $w$; it is universal for covariant representations of $P$ in the following sense: for any covariant representation $W:P\to A$ there exists a unital homomorphism $\pi_W:C^*(G,P)\to A$ such that $\pi_W(w_p)=W_p$. It follows from \eqref{covariance condition} that \[C^*(G,P)= \overline{\lsp}\{w_pw^*_q:p, q\in P\}.\]

 Nica  defined   $(G,P)$  to be \emph{amenable}  if the homomorphism $\pi_T:C^*(G,P)\to C^*_r(G,P)$ is faithful  \cite[\S4.2]{Nica1992}. He identified an equivalent condition: there exists a conditional expectation $E:C^*(G,P)\to \overline{\lsp}\{w_pw^*_p:p\in P\}$, and $(G,P)$ is amenable if and only if $E$ is faithful (that is, $E(a^*a)=0$ implies $a^*a=0$ for all $a\in C^*(G,P)$). Laca and Raeburn took this second condition as their definition of amenability \cite[Definition 3.4]{LacaRaeburn1996}.

 \section{Order-preserving maps and amenability}
 A key technique, introduced  by Laca and Raeburn in \cite[Proposition 6.6]{LacaRaeburn1996}\footnote{There is an error in the statement of \cite[Proposition 6.6]{LacaRaeburn1996}: the final line should read ``if $\calG$ is amenable then $(G,P)$ is amenable".}, is the use of an order-preserving homomorphism between two quasi-lattice ordered groups which preserves the least upper bound structure. Crisp and Laca  called such a homomorphism a \emph{controlled map} \cite{CrispLaca2007}. If $(G,P)$ and $(K,Q)$ are quasi-lattice ordered groups, $\mu:G\to K$ is a controlled map and $K$ is an amenable group, then  $(G,P)$ is amenable as a quasi-lattice ordered group by \cite[Proposition 6.6]{LacaRaeburn1996}.
Motivated by work in \cite[Appendix~A]{HuefClarkRaeburn2015} we now give a weaker definition for a controlled map. We then follow the program of \cite{HuefClarkRaeburn2015} to generalise   \cite[Proposition 6.6]{LacaRaeburn1996}.  We state this generalisation in Theorem~\ref{p-controlled map implies amenable} below; its  proof will take up the remainder of this section.

 \begin{defn}\label{expanded controlled map definition}
    Let $(G,P)$ and $(K,Q)$ be quasi-lattice ordered groups. Let $\mu:G\to K$ be an order-preserving group homomorphism. For each $k\in Q$, let $\Sigma_k$ be the set of  $\sigma\in \mu^{-1}(k)\cap P$ which are minimal in the sense that
    \[x\in \mu^{-1}(k)\cap P \text{\ and\ }x\leq\sigma\Rightarrow \sigma=x.\]
     We say $\mu$ is a \emph{controlled map} if it has the following properties:
    \begin{enumerate}
      \item\label{cm-1} For all $x,y\in G$ such that $x\vee y<\infty$ we have $\mu(x)\vee \mu(y)=\mu(x\vee y)$.
      \item\label{cm-2}  For all $k\in Q$, $\Sigma_k$ is complete in the following sense: for every $x\in \mu^{-1}(k)\cap P$ there exists $\sigma\in \Sigma_k$ such that $\sigma\leq x$.
      \item\label{cm-3}  For all $k\in Q$ and  $\sigma,\tau\in\Sigma_k$ we have $\sigma\vee\tau<\infty\Rightarrow\sigma=\tau$.
    \end{enumerate}
  \end{defn}

  \begin{thm}\label{p-controlled map implies amenable}
    Let $(G,P)$ and $(K,Q)$ be quasi-lattice ordered groups. Suppose that $\mu:G\to K$ is a controlled map. If $K$ is an amenable group and $\big(\mu^{-1}(e),\mu^{-1}(e)\cap P\big)$ is an amenable quasi-lattice ordered group, then $(G,P)$ is amenable.
  \end{thm}

We start by showing that the kernel of a controlled map is a quasi-lattice ordered group.

\begin{lemma}
  Let $(G,P)$ and $(K,Q)$ be quasi-lattice ordered groups, and suppose that $\mu:G\to K$ is a controlled map.  Then $\big(\mu^{-1}(e),\mu^{-1}(e)\cap P\big)$ is a quasi-lattice ordered group.
\end{lemma}
\begin{proof}
  It is clear that $\mu^{-1}(e)$ is a subgroup of $G$  and that$\mu^{-1}(e)\cap P$ is a unital semigroup. Suppose that $x,y\in \mu^{-1}(e)$ have a common upper bound $z\in \mu^{-1}(e)\cap P$.  We know that $z$ is a common  upper bound for $x,y$ in $P$, and hence $x\vee y$ exists in $P$ and $x\vee y\leq z$. Now $\mu(x\vee y)=\mu(x)\vee \mu(y)=e$, and hence $x\vee y\in \mu^{-1}(e)\cap P$. Thus $\big(\mu^{-1}(e),\mu^{-1}(e)\cap P\big)$ is a quasi-lattice ordered group.
\end{proof}

  To prove Theorem~\ref{p-controlled map implies amenable} we will show that the conditional expectation \[E:C^*(G,P)\to \overline{\lsp}\{w_pw^*_p:p\in P\}\] is faithful. We will use the amenability of $K$ to construct a faithful conditional expectation $\Psi^\mu:C^*(G,P)\to\clsp\{w_pw^*_q:\mu(p)=\mu(q)\}$, and then show that $E$ is faithful when restricted to $\im \Psi^\mu$. To construct $\Psi^\mu$ we follow  the method of \cite[Lemma 6.5]{LacaRaeburn1996}  which uses a coaction.

Let $G$ be a discrete group and let $A$ be a unital $C^*$-algebra. Let $$\delta_G:C^*(G)\to C^*(G)\otimes_{\min} C^*(G)$$ be the comultiplication of $G$ which is characterised by $\delta_G(u_g)=u_g\otimes u_g$ for  $g\in G$.
  A \emph{coaction} of $G$ on  $A$ is a unital homomorphism $\delta:A\to A\otimes_{\min} C^*(G)$ such that
  $$(\delta\otimes \id)\circ\delta=(\id\otimes \delta_G)\circ \delta.$$
   We say that $\delta$ is \emph{nondegenerate} if  $\delta(A)(1\otimes C^*(G))=A\otimes_{\min} C^*(G)$.
\begin{lemma}\label{faithful coaction exists}
   Let $(G,P)$ be a quasi-lattice ordered group. Suppose that there exists a group $K$ and a homomorphism $\mu:G\to K$.
  Then there exists an injective coaction $$\delta_\mu:C^*(G,P)\to C^*(G,P)\otimes_{\min} C^*(K)$$ characterised by $\delta_\mu(w_p)=w_p\otimes u_{\mu(p)}$ for all $p\in P$.
  \end{lemma}

  \begin{proof}
Let  $W:P\to C^*(G,P)\otimes_{\min}C^*(K)$ be characterised by $W_p=w_p\otimes u_{\mu(p)}$. We will show that $W$  is a covariant representation, and then take $\delta_\mu:=\pi_W$. Unitaries are isometries and hence $W_p$ is isometric for all $p\in P$. Observe that $W_e=w_e\otimes u_{\mu(e)}=1\otimes 1$, and
 $$W_pW_q=w_pw_q\otimes u_{\mu(p)}u_{\mu(q)}=w_{pq}\otimes u_{\mu(pq)}=W_{pq}\text{ for all $p,q\in P$.}$$ Thus $W$ is an isometric representation. To prove $W$ is covariant, we fix $x,y\in P$ and compute:
    \begin{align*}
      W_xW^*_xW_yW^*_y&=w_xw^*_xw_yw^*_y\otimes u_{\mu(x)}u^*_{\mu(x)}u_{\mu(y)}u^*_{\mu(y)}\\
      &=\begin{cases}w_{x\vee y}w^*_{x\vee y}\otimes 1&\text{if $x\vee y<\infty$}\\0\otimes 1&\text{otherwise}\end{cases}\\
      &=\begin{cases}w_{x\vee y}w^*_{x\vee y}\otimes w_{\mu(x\vee y)}w^*_{\mu(x\vee y)}&\text{if $x\vee y<\infty$}\\0&\text{otherwise}\end{cases}\\
      &=W_{x\vee y}W^*_{x\vee y}.
    \end{align*}
    Thus $W$ is a covariant representation of $P$.
    By the universal property of $C^*(G,P)$, there exists a  homomorphism $\delta_\mu:=\pi_W$, which has the desired properties. Since $W_e=1\otimes 1$ it follows that $\delta_\mu$ is unital.

     To prove the  comultiplication identity,  we compute on  generators: for $p,q\in P$ we have
    \begin{align*}
      ((\delta_\mu\otimes\id)\circ \delta_\mu)(w_pw^*_q)&=(\delta_\mu\otimes\id)(w_pw^*_q\otimes u_{\mu(pq^{-1})})\\
      &=\delta_\mu(w_pw^*_q)\otimes \id(u_{\mu(pq^{-1})})\\
      &=w_pw^*_q\otimes u_{\mu(pq^{-1})}\otimes u_{\mu(pq^{-1})}\\
      &=w_pw^*_q\otimes \delta_K(u_{\mu(pq^{-1})})\\
      &=\id\otimes\delta_K(w_pw^*_q\otimes u_{\mu(pq^{-1})})\\
      &=((\id\otimes\delta_K)\circ \delta_\mu)(w_pw^*_q).
    \end{align*}
    Hence $(\delta_\mu\otimes\id)\circ\delta_\mu=(\id\otimes\delta_K)\circ \delta_\mu$.  Thus $\delta_\mu$ is a coaction.

To show that $\delta_\mu$ is injective, let $\pi:C^*(G,P)\to B(H)$ be  a faithful  representation. We will show that   $\pi$ can be written as a composition of $\delta_\mu$ and another representation.
Let $\epsilon:C^*(K)\to\C$ be the trivial representation on $\C$ such that $\epsilon(u_k)=1$ for all $k\in K$. By the properties of the minimal tensor product (see \cite[Proposition B.13]{RaeburnWilliams1998}) there exists a homomorphism $$\pi\otimes \epsilon:C^*(G,P)\otimes_{\min} C^*(G)\to B(H)\otimes\C=B(H).$$
    Since     $$(\pi\otimes \epsilon)\circ \delta_\mu(w_p)=(\pi\otimes \epsilon)(w_p\otimes u_{\mu(p)})=\pi(w_p), $$
    we have   $\pi=(\pi\otimes \epsilon)\circ \delta_\mu$.  Now suppose that $\delta_\mu(a)=0$ for some $a\in C^*(G,P)$. Then $0=(\pi\otimes \epsilon)\circ \delta_\mu(a)=\pi(a)$. Since $\pi$ is faithful, $a=0$. Hence $\delta_\mu$ is injective.

    To prove that $\delta_\mu$ is a nondegenerate coaction we must show that $$\delta_\mu(C^*(G,P))(1\otimes C^*(K))=C^*(G,P)\otimes_{\min} C^*(K).$$ It suffices to show that we can get the spanning elements $w_pw^*_q\otimes u_k$, and this is easy: $$\delta_\mu(w_pw^*_q)(1\otimes u_{\mu(qp^{-1}) k})=w_pw^*_q\otimes u_{\mu(pq^{-1})}(1\otimes u_{\mu(qp^{-1}) k})=w_pw^*_q\otimes u_k.$$ Thus $\delta_\mu$ is nondegenerate.
 \end{proof}

 Let $\lambda$ be the left-regular representation of a discrete group $K$. There is a trace $\tau$ on $C^*(K)$ characterised by
 \begin{equation*}
 \tau(u_k)= (\lambda_k\epsilon_e\mid \epsilon_e)=\begin{cases}1 &\text{if $k=e$}\\ 0&\text{otherwise}.  \end{cases}
 \end{equation*}
 It is well-known that if $K$ is an amenable group, then $\tau$ is faithful.

  \begin{lemma}\label{conditional expectation construction}  Let $(G,P)$ be a quasi-lattice ordered group. Suppose that there exist a group $K$ and a homomorphism $\mu:G\to K$.
Let $$\delta_\mu:C^*(G,P)\to C^*(G,P)\otimes_{\min} C^*(K)$$ be the coaction of Lemma~\ref{faithful coaction exists}.
Then $$\Psi^\mu:=(\id\otimes \tau)\circ\delta_\mu$$ is a conditional expectation of $C^*(G,P)$ with range $\clsp\{w_pw_q^*: \mu(p)=\mu(q)\}$. If $K$ is an amenable group,  then $\Psi^\mu$ is faithful.
\end{lemma}

\begin{proof}
Since $\id\otimes \tau$ and $\delta_\mu$ are linear and norm decreasing,  so is $\Psi^\mu$. Since $\Psi^\mu(w_e)=1$ the norm of $\Psi^\mu$ is $1$.   We have
\begin{equation}\label{expectation construction eqn}
 \Psi^\mu(w_pw^*_q)=\begin{cases}w_pw^*_q&\text{if $\mu(p)=\mu(q)$}\\0&\text{otherwise,}\end{cases}
\end{equation}
and hence $\Psi^\mu\circ\Psi^\mu=\Psi^\mu$.  Thus $\Psi^\mu$ is a conditional expectation by \cite{Tomiyama1957}.

From \eqref{expectation construction eqn} we see that $\overline{\lsp}\{w_pw^*_q:\mu(q)=\mu(p)\}\subseteq \im \Psi^\mu$. To show the reverse inclusion, fix $b\in \im \Psi^\mu$, say $b=\Psi^\mu(a)$ for some $a\in C^*(G,P)$. Also  fix $\epsilon>0$. There exists a finite subset $F\subseteq P\times P$  such that $\|a-\sum_{(p,q)\in F}\lambda_{p,q}w_{p}w^*_{q}\|<\epsilon$.
Since $\Psi^\mu$ is linear and norm-decreasing,
\begin{align*}
\epsilon&>\Big\|a-\sum_{(p,q)\in F}\lambda_{p,q}w_{p}w^*_{q}\Big\|
\geq \Big\|\Psi^\mu\Big(a-\sum_{(p,q)\in F}\lambda_{p,q}w_{p}w^*_{q}\Big)\Big\|\\
&
=\Big\|\Psi^\mu(a)-\Psi^\mu\Big(\sum_{(p,q)\in F}\lambda_{p,q}w_{p}w^*_{q}\Big)\Big\|  =\Big\|b-\sum_{(p,q)\in F\text{, }\mu(p)=\mu(q)}\lambda_{p,q}w_{p}w^*_{q}\Big\|.
\end{align*}
 Thus $b\in  \overline{\lsp}\{w_pw^*_q:\mu(q)=\mu(p)\}$, and  $\im \Psi^\mu=\overline{\lsp}\{w_pw^*_q:\mu(q)=\mu(p)\}$.

 Now suppose that $K$ is amenable. To see that $\Psi^\mu$ is faithful, we follow the proof of  \cite[Lemma 6.5]{LacaRaeburn1996}. Let  $a\in  C^*(G,P)$ and suppose that  $\Psi^\mu(a^*a)=0$. Let $f$ be an arbitrary state on $C^*(G,P)$. Then
    \begin{align*}
           0&=f(\Psi^\mu(a^*a))
           =f\circ(\id\otimes \tau)\circ \delta_\mu(a^*a)\\
           &=(f\otimes \tau)\circ \delta_\mu(a^*a)
           =\tau\circ(f\otimes \id)\circ\delta_\mu(a^*a).
          \end{align*}
 Since $K$ is amenable, $\tau$ is faithful. Hence $(f\otimes \id)\circ\delta(a^*a)=0$. This  implies that for all states $f$ on $C^*(G,P)$ and states $g$ on $C^*(K)$, $$g\circ (f\otimes \id)\circ\delta_\mu(a^*a)=(f\otimes g)\circ\delta_\mu(a^*a)=0.$$

 To see that $\delta_\mu(a^*a)=0$, let $\pi_1:C^*(G,P)\to H_1$ and $\pi_2:C^*(K)\to H_2$ be faithful representations. Then $\pi_1\otimes\pi_2$ is a faithful representation of $C^*(G,P)\otimes_{\min}C^*(K)$ on $B(H_1\otimes H_2)$ by \cite[Corollary~B.11]{RaeburnWilliams1998}.  Fix  unit vectors $h\in H_1$, $k\in H_2$.  There exists a state $f_h\otimes f_k$ on $C^*(G,P)\otimes_{\min}C^*(K)$ defined by $$f_h\otimes f_k(x)=(\pi_1\otimes\pi_2(x)(h\otimes k)\mid h\otimes k).$$
Since $(f\otimes g)\circ\delta_\mu(a^*a)=0$ for all states $f$ of $C^*(G, P)$ and $g$ of $C^*(K)$, we have
\begin{align*}
  0&=f_h\otimes f_k(\delta_\mu(a^*a))\\
  &=(\pi_1\otimes\pi_2(\delta_\mu(a^*a))(h\otimes k)\mid h\otimes k)\\
  &=(\pi_1\otimes\pi_2(\delta_\mu(a))(h\otimes k)\mid \pi_1\otimes\pi_2(\delta_\mu(a)) h\otimes k)\\
  &=\|\pi_1\otimes\pi_2(\delta_\mu(a))(h\otimes k)\|^2.
\end{align*}
Hence $\pi_1\otimes\pi_2(\delta_\mu(a^*a))=0$. Since $\pi_1\otimes \pi_2$ is faithful, $\delta_\mu(a^*a))=0$.
 But $\delta_\mu$ is injective, and hence    $a=0$, and  $\Psi^\mu$ is faithful.
\end{proof}

Next we investigate the structure of
\[
\im\Psi^\mu=\clsp\{w_pw_q^*:\mu(p)=\mu(q)\}.
\]
In the statement of the next lemma, note that we can view $C^*\big(\mu^{-1}(e),\mu^{-1}(e)\cap P\big)$ as a $C^*$-subalgebra of $C^*(G,P)$.

  \begin{lemma}
     Let $(G,P)$ and $(K,Q)$ be quasi-lattice ordered groups, and suppose that  $\mu:G\to K$ a controlled map. Let $k\in Q$, and let $F$ be a finite subset of $\Sigma_k$. Let $$B_{k,F}:=\lsp\{w_\sigma Dw^*_\tau:\sigma,\tau\in F,D\in C^*\big(\mu^{-1}(e),\mu^{-1}(e)\cap P\big)\}.$$
     Then $B_{k,F}$ is a closed $C^*$-subalgebra of $\clsp\{w_pw_q^*:\mu(p)=\mu(q)\}$.
  \end{lemma}

  \begin{proof}  To see that $B_{k,F}$ is contained in  $\clsp\{w_pw_q^*:\mu(p)=\mu(q)\}$, it suffices to consider $D$ of the form $w_\alpha w_\beta^*$ where $\mu(\alpha)=\mu(\beta)=e$. Then  $w_\sigma Dw^*_\tau=w_{\sigma\alpha}w_{\tau\beta}^*$ and $\mu(\sigma\alpha)=\mu(\sigma)=k=\mu(\tau)=\mu(\tau\beta)$. Thus $B_{k,F}$ is a subalgebra of $\clsp\{w_pw_q^*:\mu(p)=\mu(q)\}$.

 We will prove the lemma by showing that $B_{k,F}$ is  isomorphic to
    \[
    M_{F}(\C)\otimes C^*\big(\mu^{-1}(e),\mu^{-1}(e)\cap P\big).
    \]
By Definition~\ref{expanded controlled map definition}\eqref{cm-3}, the elements of $F$ have no common upper bound unless they are equal. So
    $$(w_{\sigma}w^*_{\tau})(w_{\gamma}w^*_{\delta})=w_{\sigma\tau^{-1}(\tau\vee\gamma)}w^*_{\delta\gamma^{-1}(\tau\vee\gamma)}=\begin{cases}w_{\sigma}w^*_{\delta}&\text{if $\tau=\gamma$}\\0&\text{otherwise.}\end{cases}
    $$
    Thus $\{w_{\sigma}w^*_{\tau}:\sigma,\tau\in F\}$  is a set of matrix units in the $C^*$-algebra $\overline{B_{k,F}}$. This gives a homomorphism $\theta:M_{F}(\C)\to \overline{B_{k,F}}$ which maps the  matrix units  $\{E_{\sigma,\tau}:\sigma,\tau\in F\}$  in $M_{F}(\C)$ to $\{w_{\sigma}w^*_{\tau}:\sigma,\tau\in F\}\subset {B_{k,F}}$. It is  easy to check that the formula   $$\psi(D)=\sum_{\gamma\in F}w_{\gamma}Dw^*_{\gamma}$$ gives  a homomorphism $\psi:C^*\big(\mu^{-1}(e),\mu^{-1}(e)\cap P\big)\to B_{k,F}$. We have
    \begin{align*}
      \theta(E_{\sigma,\tau})\psi(D)&=w_{\sigma}w^*_{\tau}\sum_{\gamma\in F}w_{\gamma}Dw^*_{\gamma}\\
      &=w_{\sigma}w^*_{\tau}w_{\tau}Dw^*_{\tau}\quad\text{($w^*_{\tau}w_{\gamma}=0$ unless $\tau=\gamma$)}\\
      &=w_{\sigma}Dw^*_{\tau}\\
      &=(\sum_{\gamma\in F}w_{\gamma}Dw^*_{\gamma})w_{\sigma}w^*_{\tau}\\
      &=\psi(D)\theta(E_{\sigma,\tau}).
    \end{align*}
    Each $A\in M_{F}(\C)$ is a linear combination of the $E_{\sigma,\tau}$, and hence $\psi(D)\theta(A)=\theta(A)\psi(D)$ for all $A\in M_{F}(\C)$ and $D\in C^*\big(\mu^{-1}(e),\mu^{-1}(e)\cap P\big)$. Since the ranges of $\theta$ and $\psi$ commute, the universal property of the maximal tensor product gives a homomorphism $\theta\otimes_{\max} \psi$ of $M_{F}(\C)\otimes_{\max}C^*\big(\mu^{-1}(e),\mu^{-1}(e)\cap P\big)$ into $\overline{B_{k,F}}$.

By \cite[Theorem B.18]{RaeburnWilliams1998} 
\begin{align*}
M_{F}(\C)\otimes&_{\max}C^*\big(\mu^{-1}(e),\mu^{-1}(e)\cap P\big)\\
&=\lsp\{E_{\sigma,\tau}\otimes D: \sigma,\tau\in F \text{\ and\ } D\in C^*\big(\mu^{-1}(e),\mu^{-1}(e)\cap P\big) \},\end{align*}
 with no closure. So  the range of $\theta\otimes_{\max} \psi$ is spanned by $\theta(E_{\sigma,\tau})\psi(D)=w_{\sigma}Dw^*_{\tau}$ and hence  is $B_{k,F}$. Thus $B_{k,F}$ is a closed $C^*$-subalgebra of $\clsp\{w_pw_q^*:\mu(p)=\mu(q)\}$.
\end{proof}

Let $\{\epsilon_x:x\in P\}$ be the usual basis for $\ell^2(P)$. Let $T$ be the covariant  representation of $(G,P)$ on $\ell^2(P)$ such that $T_p\epsilon_x=\epsilon_{px}$, and let $\pi_T$ be the corresponding homomorphism of $C^*(G,P)$ onto $C^*_r(G,P)$ such that $\pi_T(w_p)=T_p$.
For $k\in Q$ we  consider the subspaces
\[
H_k:=\{\epsilon_{\gamma z}:\gamma\in\Sigma_k,z\in \mu^{-1}(e)\cap P\}.
\]

\begin{lemma}\label{isometric on H_k} Let $(G,P)$ and $(K,Q)$ be quasi-lattice ordered groups, and suppose that $\mu:G\to K$ is a controlled map.
Let $k\in Q$. Then
  \begin{enumerate}
    \item $H_k$ is invariant for $\pi_T|_{B_{k,F}}$;
    \item\label{isometric on H_k2} if $\big(\mu^{-1}(e),\mu^{-1}(e)\cap P\big)$ is amenable, then $\pi_T|_{B_{k,F}}$ is isometric on $H_k$.
  \end{enumerate}
  \end{lemma}
  \begin{proof}
    For   (1), let $\sigma,\tau\in F$ and let $x,y\in\mu^{-1}(e)$ and let $\epsilon_{\gamma z}\in H_k$. Then $w_{\sigma}w_{x}w^*_yw^*_{\tau}$ is a spanning element of
 $B_{k,F}$.  Since $\mu(\tau)=k=\mu(\gamma)$ we have
    $$
    \pi_T(w_{\sigma}w_{x}w^*_yw^*_{\tau})\epsilon_{\gamma z}=
    \begin{cases}\epsilon_{\sigma xy^{-1} z}&\text{if $\gamma=\tau$ and $y\leq z$}\\0&\text{otherwise.}\end{cases}
    $$
If  $\pi_T(w_{\sigma}w_{x}w^*_yw^*_{\tau})\epsilon_{\gamma z}=0$ we are done. Otherwise, to see that $\epsilon_{\sigma xy^{-1} z}$ is back in $H_k$, suppose that $y\leq z$. Then $y^{-1}z\in P$. Since $\sigma x\in \mu^{-1}(k)\cap P$ we have $\epsilon_{(\sigma x)(y^{-1} z)}\in H_k$. It follows that $H_k$ is invariant for $\pi_T|_{B_{k,F}}$.

    For (2) suppose that $\big(\mu^{-1}(e),\mu^{-1}(e)\cap P\big)$ is amenable.  We will show that $\pi_T|_{B_{k,F}}$ is faithful on $H_k$. Take $B=\sum_{\sigma,\tau\in F} w_{\sigma}D_{\sigma,\tau}w^*_{\tau}\in B_{k,F}$ and suppose that $\pi_T(B)|_{H_k}=0$. Fix $\gamma,\delta\in F$. Then
$$T^*_{\gamma}\pi_T(B)T_{\delta}=\pi_T(w^*_{\gamma})\pi_T(B)\pi_T(w_{\delta})=\pi_T(D_{\gamma,\delta})$$
    Since $T_{\delta}$ is an injection from $H_{e}$ to $H_k$ and $\pi_T(B)|_{H_k}=0$, it follows that  $\pi_T(B)T_{\delta}|_{H_e}=0$. Thus $\pi_T(D_{\gamma,\delta})|_{H_e}=0$.

    But the restriction
    \[
    (\pi_T|_{C^*\big(\mu^{-1}(e),\mu^{-1}(e)\cap P\big)})|_{H_e}
    \]
     is the Toeplitz representation of $\big(\mu^{-1}(e),\mu^{-1}(e)\cap P\big)$, and hence  is  faithful by amenability.  Thus $D_{\gamma,\delta}=0$. Repeating the argument finitely many times shows that all the $D_{\sigma,\tau}=0$ and hence that $B=0$. Thus $\pi_T|_{B_{k,F}}$ is faithful on $H_k$, and therefore is  isometric.
  \end{proof}

  \begin{lemma}\label{faithful on Bk}
  Let $(G,P)$ and $(K,Q)$ be quasi-lattice ordered groups, and suppose that  $\mu:G\to K$ is a controlled map.
  Let $B_k=\overline{\lsp}\{w_pw^*_q:\mu(p)=\mu(q)=k\}$. Let $\calF$ be the set of all finite sets $F\subseteq \Sigma_k$. Then $B_k=\overline{\cup_{F\in\calF}B_{k,F}}$.
If $\big(\mu^{-1}(e),\mu^{-1}(e)\cap P\big)$ is amenable, then $\pi_T|_{B_{k}}$ is isometric on $H_k$.
  \end{lemma}

  \begin{proof}
    Observe that $\calF$ is a directed set partially ordered by inclusion with $E,F\in \calF$ majorismajorizeded by $E\cup F$. If $E\subseteq F$, then $B_{k,E}\subseteq B_{k,F}$. Thus $\{B_{k,F}:F\in \calF\}$ is an inductive system with limit $\overline{\cup_{F\in\calF}B_{k,F}}$.

    For each $F\in\calF$ we have $B_{k,F}\subseteq B_k$, and $B_k$ is closed. Therefore $\overline{\cup_{F\in\calF}B_{k,F}}\subseteq B_k$. To prove the reverse inclusion it suffices to show that the spanning elements of $B_k$ are in $B_{k,F}$ for some $F$. Fix $p,q\in P$ such that $\mu(p)=\mu(q)=k$ and consider $w_pw^*_q$.  By Definition~\ref{expanded controlled map definition}\eqref{cm-2}, the set $\Sigma_k$  of minimal elements is complete, and  there exists $\sigma,\tau\in\Sigma_k$ such that $\sigma\leq p$ and $\tau \leq q$.   Hence there exists $x,y\in  P$ such that
$p=\sigma x$ and $q=\tau y$.
Thus $w_pw^*_q=w_{\sigma x}w^*_{\tau y}=w_{\sigma}(w_xw^*_y)w^*_{\tau}$ and  $w_xw^*_y\in  C^*\big(\mu^{-1}(e),\mu^{-1}(e)\cap P\big)$. Since $\{\sigma,\tau\}\in\calF$ we have  $w_pw^*_q\in B_{k,\{\sigma,\tau\}}$. Thus  $B_k\subseteq \overline{\cup_{F\in\calF}B_{k,F}}$, and equality follows.

Finally, suppose that $\big(\mu^{-1}(e),\mu^{-1}(e)\cap P\big)$ is amenable. Then $\pi_T|_{B_{k,F}}$ is isometric on $H_k$ for all $F\in\calF$ by Lemma~\ref{isometric on H_k}\eqref{isometric on H_k2}. Since $\pi_T$ is isometric on every $B_{k,F}$,  its extension to the closure is also isometric.
  \end{proof}

  Let $\calI$ be the set of all finite sets $I\subset Q$ that are closed under $\vee$ in the sense that $s,t\in I$ and $s\vee t<\infty$ implies that $s\vee t\in I$.

  \begin{lemma}\label{core is limit of CI}
  Let $(G,P)$ and $(K,Q)$ be quasi-lattice ordered groups, and suppose that  $\mu:G\to K$  is a controlled map.
  For each $I\in\calI$ let \[C_I=\overline\lsp\{w_pw^*_q:\mu(p)=\mu(q)\in I\}.\] Then $C_I$ is a $C^*$-subalgebra of $\clsp\{w_pw_q^*:\mu(p)=\mu(q)\}$, $C_I=\lsp\{B_k:k\in I\}$ and $\clsp\{w_pw_q^*:\mu(p)=\mu(q)\}=\overline{\cup_{I\in\calI}C_I}$.
  \end{lemma}
  \begin{proof}
  Fix $I\in\calI$. To see that $C_I$ is a $C^*$-subalgebra,  it suffices to show that $\lsp\{w_pw^*_q:\mu(p)=\mu(q)\in I\}$ is a $*$-subalgebra. It's clearly closed under taking adjoints. Let $p,q,r,s\in P$ such that $\mu(p)=\mu(q)\in I$ and $\mu(r)=\mu(s)\in I$. Then
  $$
  w_pw^*_qw_rw^*_s=\begin{cases}w_{pq^{-1}(q\vee r)}w^*_{sr^{-1}(q\vee r)}&\text{if $q\vee r<\infty$}\\0&\text{otherwise.}\end{cases}
  $$
    If $w_pw^*_qw_rw^*_s=0$ we are done. So suppose  that $w_pw^*_qw_rw^*_s\neq 0$. Then $q\vee r<\infty$. Since $\mu$ is a controlled map and $\mu(p)=\mu(q)$, by Definition~\ref{expanded controlled map definition}\eqref{cm-1},
    \[
    \mu(pq^{-1}(q\vee r))=\mu(q\vee r)=\mu(q)\vee \mu(r).
    \]
    Similarly, $\mu(sr^{-1}(q\vee r))=\mu(q)\vee \mu(r)$. Since $I$ is closed under $\vee$ we have  $\mu(q)\vee \mu(r)\in I$, and  hence $w_pw^*_qw_rw^*_s\in \lsp\{w_pw^*_q:\mu(p)=\mu(q)\in I\}$. It follows that  $C_I$ is a $C^*$-subalgebra.

    For each $k\in I$, we have $B_k\subseteq C_I$, and so $\lsp\{B_k:k\in I\}\subseteq C_I$. To show the reverse inclusion observe that for $w_pw^*_q\in C_I$ we have $w_pw^*_q\in B_{\mu(p)}$. Since  the finite span of closed subalgebras is closed,  $\overline\lsp\{w_pw^*_q:\mu(p)=\mu(q)\in I\}\subseteq\lsp\{B_k:k\in I\}$. Thus $C_I=\lsp\{B_k:k\in I\}$.
  \end{proof}

  \begin{prop}\label{piT is faithful on core}
  Let $(G,P)$ and $(K,Q)$ be quasi-lattice ordered groups, and suppose that $\mu:G\to K$ is a controlled map.
    If $\big(\mu^{-1}(e),\mu^{-1}(e)\cap P\big)$ is amenable, then $\pi_T$ is faithful on $\overline\lsp\{w_pw^*_q:\mu(p)=\mu(q)\}$.
  \end{prop}

\begin{proof}
By Lemma~\ref{core is limit of CI}, $\overline\lsp\{w_pw^*_q:\mu(p)=\mu(q)\}=\overline{\cup_{I\in\calI}C_I}$. Thus it suffices to show that $\pi_T$ is isometric on each $C_I$. Fix $I\in \calI$. Suppose that $\pi_T(R)=0$ for some $R\in C_I$. Then there exist $R_k\in B_k$ such that $R=\sum_{k\in I}R_k$ and then $\sum_{k\in I}\pi_T(R_k)=0$.

We claim that if $k\not\leq j$, then  $\pi_T(B_k)|_{H_j}=0$ (it then follows that $\pi_T(R_k)|_{H_j}=0$). To prove the claim, it suffices to show that $\pi_T(w_q^*)\epsilon_{\gamma z}=0$ for all $q\in \mu^{-1}(k)\cap P$ and $\epsilon_{\gamma z}\in H_k$.  We have
\[
\pi_T(w_q^*)\epsilon_{\gamma z}=T^*_q\epsilon_{\gamma z}
=\begin{cases}
\epsilon_{q^{-1}\gamma z} &\text{if $q\leq \gamma z$}\\
0&\text{otherwise.}
\end{cases}
\]
But $q\leq \gamma z$ implies $k=\mu(q)\leq \mu(\gamma z)=\mu(\gamma)=j$. So $k\not\leq j$ implies $\pi_T(w_q^*)\epsilon_{\gamma z}=0$. Hence $\pi_T(B_k)|_{H_j}=0$ if $k\not\leq j$ as claimed.

  Let $l$ be a minimal element of $I$ in the sense that $x\leq l$ implies $x=l$. Then for $k\in I$, we have $k\not\leq l$ unless $k=l$. Now
  $$
  0=\sum_{k\in I}\pi_T(R_k)|_{H_l}=\pi_T(R_l)|_{H_l}.
  $$
Since $\big(\mu^{-1}(e),\mu^{-1}(e)\cap P\big)$ is amenable, $\pi_T|_{B_l}$ is isometric on $H_l$ by Lemma~\ref{faithful on Bk}. Thus $R_l=0$.

Let $l_2$ be a minimal element of $I\backslash\{l\}$. Then we can repeat the above argument to get $R_{l_2}=0$. Since $I$ is finite, we can continue to conclude that $R=0$.
\end{proof}
We can now prove Theorem~\ref{p-controlled map implies amenable}

\begin{proof}[Proof of Theorem~\ref{p-controlled map implies amenable}]
Suppose that $K$ is an amenable  group. To see $(G,P)$ is amenable, we will show that the conditional expectation \[E:C^*(G,P)\to \overline{\lsp}\{w_pw^*_p:p\in P\}\] is faithful.
Let $\Psi^\mu$ be the conditional expectation of Lemma~\ref{conditional expectation construction}. We have
\begin{align*}
  E(\Psi^{\mu}(w_pw^*_q))&=\begin{cases}\Psi(w_pw^*_q)&\text{if $\mu(p)=\mu(q)$}\\0&\text{otherwise}\end{cases}\\
  &=\begin{cases}w_pw^*_p&\text{if $p=q$}\\0&\text{otherwise}\end{cases}\\
  &=\Psi(w_pw^*_q),
\end{align*}
and hence $E=E\circ\Psi^{\mu}$.

 Since $K$ is an amenable group, $\Psi^\mu$ is faithful  by Lemma~\ref{conditional expectation construction}. Let  $P_z\in B(\ell^2(P))$ be the orthogonal projection onto $\lsp\{\epsilon_z\}$. It is straightforward to show that the diagonal map $\Delta:B(\ell^2(P))\to B(\ell^2(P))$ given by
$$\Delta(T)=\sum_{z\in P}P_zTP_z$$
is a  conditional expectation such that $\Delta\circ \pi_T=\pi_T\circ E$ and is faithful.

Now suppose that $R\in C^*(G,P)$ and $E(R^*R)=0$. Then $E(\Psi^{\mu}(R^*R))=0$ and so $\pi_T\circ E(\Psi^{\mu}(R^*R))=0$. This gives $\Delta\circ \pi_T(\Psi^\mu(R^*R))=0$. Since $\Delta$ is faithful, it follows that $\pi_T(\Psi^\mu(R^*R))=0$. Since  $\big(\mu^{-1}(e),\mu^{-1}(e)\cap P\big)$ is amenable,  Lemma~\ref{piT is faithful on core} implies that $\pi_T$ is faithful on $\clsp\{w_pw_q^*:\mu(p)=\mu(q)\}=\im\Psi^\mu$. Thus $\Psi^\mu(R^*R)=0$, and then  $R=0$ since $\Psi^{\mu}$ is faithful.
Hence $E$ is faithful and $(G,P)$ is amenable.
\end{proof}

  \section{Quasi-lattice ordered HNN extensions}

   Let $G$ be a group, let $A$ and $B$ be subgroups of $G$,  and let $\phi:A\to B$ be an isomorphism. The group with presentation
   $${G^*=\langle G,t\mid t^{-1}at=\phi(a),a\in A\rangle}$$ is the \emph{HNN extension of G} with respect to $A,B$ and $\phi$. For every HNN extension $G^*$  the \emph{height map} is the homomorphism $\theta:G^*\to \Z$ such that $\theta(g)=0$ for all $g\in G$ and $\theta(t)=1$.

\begin{example*}
 Let  $c,d\in \Z$. The Baumslag-Solitar group
  $$\BS(c,d)=\langle x,t\mid t^{-1}x^dt=x^c\rangle =\langle x,t\mid tx^c=x^dt\rangle$$
   is an HNN extension of $\Z$ with respect to $A=d\Z$, $B=c\Z$ and  $\phi:A\to B$ given by $\phi(dn)=cn$ for all $n\in \Z$. Then $\Z^*$ satisfies the relation $t^{-1}dt=\phi(d)=c$.   Let $\BS(c,d)^+$ be the subsemigroup of $\BS(c,d)$ generated by $x$ and $t$.
Spielberg showed in \cite[Theorem~2.11]{Spielberg2012} that $(\BS(c,d), \BS(c,d)^+)$ is quasi-lattice ordered for all $c,d> 0$; he also proved in \cite[Lemma 2.12]{Spielberg2012} that $(\BS(c,-d),\BS(c,-d)^+)$ is not quasi-lattice ordered unless $c=1$.

\end{example*}

To work with  an HNN extension we use a normal form for its elements from \cite[Theorem 2.1]{LyndonSchupp1977}.
   We choose $X$ to be a complete set of left coset representatives for $G/A$ such that $xA\neq x'A$ for  $x\neq x'\in X$. Similarly, choose  a complete set $Y$ of left coset representatives for $G/B$. Then a  (right) \emph{normal form relative to $X$ and $Y$} of $g\in G$ is a product \[g=g_0t^{\epsilon_1}g_1t^{\epsilon_2}\ldots g_{n-1}t^{\epsilon_n}g_n\] where:
    \begin{enumerate}
      \item $g_n$ is an arbitrary element of $G$.
      \item If $\epsilon_i=1$, then $g_{i-1}$ is an element of $X$
      \item If $\epsilon_i=-1$, then $g_{i-1}$ is an element of $Y$.
    \end{enumerate}
By \cite[Theorem 2.1]{LyndonSchupp1977}, for every choice of complete  left coset representatives $X$ and $Y$, each  $g\in G^*$ has a unique normal form.

Our goal is to generalise the properties of the Baumslag-Solitar group with $c, d>0$ to construct quasi-lattice ordered HNN extensions of other quasi-lattice ordered groups.

  Let $(G,P)$ be a quasi-lattice ordered group. Let $G^*$ be the HNN extension of $G$ with respect to subgroups $A$ and $B$ with an isomorphism $\phi:A\to B$. Let $P^*$ be the subsemigroup of $G^*$ generated by $P$ and $t$. In general, $(G^*,P^*)$ is not a quasi-lattice ordered group. For example, if $c>1$, then $(\BS(c,-d),\BS(c,-d)^+)$ is not quasi-lattice ordered by \cite[Lemma 2.12]{Spielberg2012}. We need some conditions on our subgroups $A$ and $B$ and on the isomorphism $\phi$ which ensure that $(G^*,P^*)$ is quasi-lattice ordered.

  There are two reasons why $(\BS(c,d),\BS(c,d)^+)$ is easy to work with. The first is that there are   natural choices of coset representatives: $\{0,\ldots,d-1\}$ for $A=d\Z$ and $\{0,\ldots,c-1\}$ for $B=c\Z$. The second is that the subgroup isomorphism $\phi:md\mapsto mc$ takes positive elements to positive elements. In particular, using this choice of representatives, every element $\omega\in\BS(c,d)^+$ has a unique normal form
  $$\omega=x^{m_0}tx^{m_1}t\ldots x^{m_{n-1}}tx^{m_n}$$ where $0\leq m_i<d$ for $i<n$ and $m_n\in \N$.  This choice of coset representatives is associated to the division algorithm on $\N$: for every $n\in \N$ we can uniquely write $n=md+r$ for some $m\in \N$ and $0\leq r\leq d-1$.

In general, for $G^*$ we would like a natural choice of coset representatives for $G/A$ and $G/B$ so that every element of $P^*$ has a unique normal form that is a sequence of elements in $P$ and $t$.

  \begin{thm}\label{(G*,P*) is left ql}
    Let $(G,P)$ be a quasi-lattice ordered group with subgroups $A$ and $B$. Suppose that:
    \begin{enumerate}
      \item There is an isomorphism $\phi:A\to B$ such that $\phi(A\cap P)=B\cap P$;
      \item Every left coset $gA\in G/A$ such that $gA\cap P\neq \emptyset$ has a minimal coset representative $p\in P$: $q\in gA\cap P \Rightarrow p\leq q$;
      \item For every $x,y\in B$, $x\vee y<\infty\Rightarrow x\vee y\in B$.
    \end{enumerate}
    Let $G^*=\langle G,t\mid t^{-1}at=\phi(a),a\in A\rangle$ be the HNN extension of $G$ and let $P^*$ be the subsemigroup of $G^*$ generated by $\{P,t\}$. Then $(G^*,P^*)$ is quasi-lattice ordered.
  \end{thm}
    Before we can prove Theorem~\ref{(G*,P*) is left ql}, we need to prove two lemmas. The first shows that elements of $P^*$ are guaranteed to have normal forms made up of elements of $P$ and $t$ if and only if condition (2) of Theorem~\ref{(G*,P*) is left ql} holds. The second is a technical lemma which we will use several times in Theorem~\ref{(G*,P*) is left ql} and in  later proofs.

  \begin{lemma}\label{right lower bound gives unique normal form in P}
    Suppose that $(G,P)$ is a quasi-lattice ordered group with subgroups $A$ and $B$. Suppose that  $\phi:A\to B$  is a group isomorphism such that $\phi(A\cap P)=B\cap P$.
    Let $G^*=\langle G,t\mid t^{-1}at=\phi(a),a\in A\rangle$ be the corresponding HNN extension of $G$ and let $P^*$ be the subsemigroup of $G^*$ generated by $P\cup\{t\}$.  Let $$L_A:=\{p\in P:  q\in pA\cap P\Rightarrow p\leq q\}.$$ The following two statements are equivalent:
    \begin{enumerate}
      \item Every left coset $gA\in G/A$ such that $gA\cap P\neq \emptyset$ has a minimal coset representative $p\in P$;
      \item There exists a complete set $X$ of left coset representatives such that $L_A\subseteq X$ and every $\alpha\in P^*$ has normal form
      \begin{equation}\label{normal form}
      \alpha=p_0tp_1t\ldots p_{n-1}tp_n  \text{\ where $p_i\in L_A$ for all $0\leq i< n$, $p_n\in P$.}
      \end{equation}
    \end{enumerate}
  \end{lemma}

  \begin{proof}
  Assume (1).
   Choose a complete set $X$ of coset representatives which contains $L_A$.  Let $\alpha\in P^*$. If $\theta(\alpha)=0$ then $\alpha\in P$, and $\alpha$ has form \eqref{normal form} trivially.

 We  proceed by induction on $\theta(\alpha)\geq 1$. Suppose that $\theta(\alpha)=1$. We may write $\alpha=q_0tq_1$ for some $q_0,q_1\in P$. Then $q_0A\cap P\neq \emptyset$, and  there exists $p_0\in L_A$ such that $p_0A=q_0A$ and $p_0\leq q_0$. Thus$p_0^{-1}q_0\in P\cap A$. Hence $q_0 =p_0a$ for some $a\in A\cap P$. Thus $\alpha$ has normal form
 $$
 \alpha=p_0atq_1=p_0t\phi(a)q_1.
 $$
Since $\phi(A\cap P)=B\cap P$ we have $\phi(a)\in P$ and so $X$ satisfies (2).

    Suppose that all $\alpha$ with $1\leq \theta(\alpha)\leq k$ have normal form \eqref{normal form}. Consider $\alpha$ with $\theta(\alpha)=k+1$. We write $$\alpha=q_0tq_1t\ldots tq_ktq_{k+1}.$$
    By assumption, we can write the first $2k+2$ terms of $\alpha$ in normal form $$p_0tp_1t\ldots p_{k-1}tr_k$$ where $p_i\in L_A$ for $0\leq i<k$ and $r_k\in P$. There exists $p_k\in L_A$ such that $r_kA=p_kA$ and $p_k\leq r_k$. As above, we can write $r_k=p_ka$ for some $a\in A\cap P$. Then
    \begin{align*}
      \alpha&=p_0tp_1t\ldots tr_ktq_{k+1}\\
      &=p_0tp_1t\ldots tp_katq_{k+1}\\
      &=p_0tp_1t\ldots tp_kt\phi(a)q_{k+1}.
    \end{align*}
We set $p_{k+1}=\phi(a)q_{k+1}$, which is in $P$ because $\phi(a)$ is.    Then $\alpha=p_0tp_1t\ldots tp_ktp_{k+1}$ has form \eqref{normal form}. By  induction, every $\alpha$ has normal form \eqref{normal form}. This implies (2).

For (2) $\Rightarrow$ (1) we argue by contradiction: we will assume (2) holds but (1) doesn't.  Let $X$  be a set of coset representatives satisfying (2), and suppose that there exists some coset $gA$ such that $gA\cap P\neq\emptyset$ which has no minimal coset representative. Let $p\in X$ be the coset representative of $gA$. Since $p$ is not  minimal, there exists $q\in gA\cap P$ with $p\not\leq q$. Thus $p^{-1}q\not\in P$. Consider $qt\in P^*$ in normal form:
\begin{equation}\label{qt in normal form}
qt=pp^{-1}qt=pt\phi(p^{-1}q).
\end{equation}
 Since $\phi(A\cap P)=B\cap P$ we have $\phi(p^{-1}q)\not \in P$. So \eqref{qt in normal form}  is not the  normal form \eqref{normal form}, a contradiction. Thus $(2)\Rightarrow (1)$.
  \end{proof}

\begin{lemma}\label{PP-1-closed is vee closed}
  Let $(G,P)$ be a quasi-lattice ordered group and let $B$ be a subgroup of $G$. Suppose that for every $x,y\in B$, $x\vee y<\infty\Rightarrow x\vee y\in B$. Then for all $x\in B\cap PP^{-1}$, there exist $\mu,\nu\in B\cap P$ such that $x=\mu\nu^{-1}$ and for all $p,q\in P$ with $pq^{-1}=x$ we have $\mu\leq p$ and $\nu\leq q$.
  \end{lemma}

The lemma says that if $x\in PP^{-1}\cap B$, then the minimal elements of \eqref{Crisp Laca} must also be contained in $B$. In particular,  if $\phi(A\cap P)=B\cap P$, then $\phi(x)\in PP^{-1}$.

  \begin{proof}[Proof of Lemma~\ref{PP-1-closed is vee closed}]
Fix $x\in B\cap PP^{-1}$. Say $x=st^{-1}$ with $s,t\in P$. Then $x^{-1}s\in P$ and $x\leq s$. Also $e\leq s$, and so $x\vee e<\infty$. Since $e,x\in B$ we get $x\vee e\in B$.  Let  $\mu=x\vee e$ and $\nu=x^{-1}(x\vee e)$.  Then $\mu\nu^{-1}=x\vee e(x^{-1}(x\vee e))^{-1}=x$.

Fix  $p,q\in P$ such that  $x=pq^{-1}$.
Then $x^{-1}p=(pq^{-1})^{-1}p=q$, and so $x\leq p$. Therefore $\mu=x\vee e\leq p$. Now $\mu^{-1}p\in P$, and then $\nu^{-1}q=\mu^{-1}p\in P$ gives $\nu\leq q$.
  \end{proof}

We can now prove Theorem~\ref{(G*,P*) is left ql}. Its proof is based on \cite[Theorem~2.11]{Spielberg2012}, and our presentation is  helped by  Emily Irwin's treatment of \cite[Theorem~2.11]{Spielberg2012} in her University of Otago Honours thesis.

  \begin{proof}[Proof of Theorem~\ref{(G*,P*) is left ql}]
 Fix $x\in P^*P^{*-1}$.  We shall prove that there exist $\mu,\nu\in P^*$ with $x=\mu\nu^{-1}$ such that whenever $\gamma\delta^{-1}=x$ we have $\mu\leq \gamma$ and $\nu\leq \delta$ (see \eqref{Crisp Laca}).

 Choose $\alpha,\beta\in P^*$ such that $x=\alpha\beta^{-1}$. Choose a complete set $X$ of left coset representatives  of $G/A$ that contains
 \[
 L_A:=\{p\in P: q\in pA\cap P\Rightarrow p\leq q\}.
 \]
  By Lemma~\ref{right lower bound gives unique normal form in P} we can write $\alpha$ and $\beta$ in unique normal form:
    $$\alpha=p_0tp_1t\ldots tp_mtr \text{ where $p_i\in L_A$ and $r\in P$};$$
    $$\beta=q_0tq_1t\ldots tq_nts \text{ where $q_i\in L_A$ and $s\in P$}.$$
    Now $x=\alpha\beta^{-1}$ is equal to
    $$\alpha\beta^{-1}=p_0tp_1t\ldots tp_mtrs^{-1}t^{-1}q_n^{-1}t^{-1}\ldots t^{-1}q_0^{-1}.$$
First we look for initial cancellations in the middle of $\alpha\beta^{-1}$: if $rs^{-1}\in B$, then we can replace $trs^{-1}t^{-1}$ with $\phi^{-1}(rs^{-1})$.  By assumption (3),  Lemma~\ref{PP-1-closed is vee closed} applies and there exist $b_1,b_2\in P\cap B$ such that $rs^{-1}=b_1b_2^{-1}$. Then
\[
\phi^{-1}(rs^{-1})=\phi^{-1}(b_1b_2^{-1})=\phi^{-1}(b_1)\phi^{-1}(b_2)^{-1}.
\]
  Since $\phi(A\cap P)=B\cap P$ we have $\phi^{-1}(b_1)\phi^{-1}(b_2)^{-1}\in PP^{-1}$. Then
\begin{equation}\label{alpha beta inverse}
x=\alpha\beta^{-1}=p_0tp_1t\ldots tp_m\phi^{-1}(b_1)\phi^{-1}(b_2)^{-1}q_n^{-1}t^{-1}\ldots t^{-1}q_0^{-1}.
\end{equation}
 We can repeat this process until there are no more cancellations available in the middle, and so we assume this is the case for the expression \eqref{alpha beta inverse}.
This gives the following cases:
    \begin{enumerate}
      \item there are no $t$ and no more $t^{-1}$,
      \item there are no more $t^{-1}$,
      \item there are no more $t$,
      \item  there are $t$ and $t^{-1}$, and then the term with $t$ to the left and $t^{-1}$ to its right is not in $B$.
    \end{enumerate}
   In each case, we will write down  our candidates for $\mu$ and $\nu$ and prove that they are the required minimums.

(1) Suppose that after initial cancellations, there are  no more $t$ and no more $t^{-1}$.  Then $\alpha\beta^{-1}=p_0q_0^{-1}$ is already in normal form. By \eqref{Crisp Laca}  there exist $\sigma,\tau\in P$ such that $p_0q_0^{-1}=\sigma\tau^{-1}$ and for all $c,d\in P$ such that $cd^{-1}=\sigma\tau^{-1}$ we have $\sigma\leq c$ and $\tau\leq d$. So we write $x=\sigma\tau^{-1}$ and choose as our candidates $\mu=\sigma$ and $\nu=\tau$.

Let $\gamma,\delta\in P^*$ such that $x=\gamma\delta^{-1}$.
Let $\theta$ be the height map. Then  $\theta(x)=0$ and hence $\theta(\gamma)=\theta(\delta)$. We will prove that $\mu\leq \gamma$ and $\nu\leq \delta$ by induction on $\theta(\gamma)$.

For $\theta(\gamma)=0$ we have $\gamma,\delta\in P$ , and then  $\mu=\sigma\leq\gamma$ and $\nu=\tau\leq \delta$.
Let $k\geq 0$ and suppose that for all $\gamma,\delta\in P^*$ such that $\theta(\gamma)=\theta(\delta)=k$ and $x=\gamma\delta^{-1}$ we have $\mu\leq \gamma$ and $\nu\leq \delta$.
    Now consider $\gamma,\delta\in P^*$ such that $x=\gamma\delta^{-1}$ and $\theta(\gamma)=\theta(\delta)=k+1$. We write
    $\gamma=m_0t\ldots m_ktm_{k+1}$ and $\delta=n_0t\ldots n_ktn_{k+1}$ in normal form where $m_i,n_i\in L_A$ for $0\leq i\leq k$ and $m_{k+1}, n_{k+1}\in P$.
Next we reduce $x=\gamma\delta^{-1}$ towards normal form. We have
$$x=\gamma\delta^{-1}=m_0t\ldots m_ktm_{k+1}n_{k+1}^{-1}t^{-1}n_k\ldots tn_0^{-1}.$$
Since $x$ has a unique normal form with no $t$ or $t^{-1}$, there must be some cancellation. Since the $m_i, n_i\in L_A$ for $0\leq i\leq k$, the cancellation must occur across $m_{k+1}n_{k+1}^{-1}$. So $m_{k+1}n_{k+1}^{-1}\in B$ and $tm_{k+1}n_{k+1}^{-1}t^{-1}=\phi^{-1}(m_{k+1}n_{k+1}^{-1})$, and
    \begin{align*}
      x&=\gamma\delta^{-1}=m_0t\ldots tm_k(tm_{k+1}n_{k+1}^{-1}t^{-1})n_kt^{-1}\ldots tn_0^{-1}\\
      &=m_0t\ldots tm_k\phi^{-1}(m_{k+1}n_{k+1}^{-1})n_kt^{-1}\ldots tn_0^{-1}
    \end{align*}
 By assumption (3), Lemma~\ref{PP-1-closed is vee closed} applies and there exists $b_m,b_n\in B\cap P$ such that  $m_{k+1}n_{k+1}^{-1}=b_mb_n^{-1}$  and  $b_m\leq m_{k+1}$ and $b_n\leq n_{k+1}$. Then
   \[
  x=m_0t\ldots tm_k\phi^{-1}(b_m)\phi^{-1}(b_n^{-1})n_kt^{-1}\ldots tn_0^{-1}.
  \]
  Since $\phi(A\cap P)=B\cap P$ we have $m_{k+1}\phi^{-1}(b_m),n_{k+1}\phi^{-1}(b_n)\in P$.
    But now we have $\gamma'=m_0t\ldots tm_k\phi^{-1}(b_m)$ and $\delta'=n_0t\ldots tn_k\phi^{-1}(b_n)$ such that $\gamma'(\delta')^{-1}=x$ and $\theta(\gamma)=\theta(\delta')=k$. By our induction hypothesis we have $\mu\leq \gamma'$ and $\nu\leq \delta'$.

To show that $\mu\leq \gamma$ we compute:
    \begin{align*}
      \gamma&=m_0t\ldots tm_ktm_{k+1}\\
      &=m_0t\ldots tm_ktb_mb_m^{-1}m_{k+1}\\
      &=m_0t\ldots tm_k\phi^{-1}(b_m)tb_m^{-1}m_{k+1}\quad\text{(replacing $tb_m$ with $\phi^{-1}(b_m)t$)}\\
      &=\gamma'tb_m^{-1}m_{k+1}.
    \end{align*}
Since  $b_m^{-1}m_{k+1}\in P$  we have $tb_m^{-1}m_{k+1}\in P^*$. Now $(\gamma')^{-1}\gamma\in P^*$ and $\gamma'\leq \gamma$. Since $\mu\leq \gamma'$ this gives $\mu\leq \gamma$.
 Similarly, $\delta=\delta'tb_n^{-1}n_{k+1}$ and so $\nu\leq \delta$.
    By induction, for all $\gamma,\delta\in P^*$ such that $x=\gamma\delta^{-1}$ we have $\mu\leq\gamma$ and $\nu\leq \delta$.

    (2) Suppose that after the initial cancellations there are no more $t^{-1}$ left. Then we have $x$ in normal form:
    $$x=\alpha\beta^{-1}=p_0tp_1t\ldots tp_{m-n-1}trs^{-1}$$
    where $m\geq n$.
    By \eqref{Crisp Laca}  there exist $\sigma,\tau\in P$ such that $rs^{-1}=\sigma\tau^{-1}$ and $\sigma\leq r$ and $\tau\leq s$. So
    \begin{equation}\label{eq-x}
    x=p_0tp_1t\ldots tp_{m-n-1}t\sigma\tau^{-1}.
    \end{equation}
     Our candidates are
    $$\mu=p_0tp_1t\ldots tp_{m-n-1}t\sigma\quad\text{and}\quad \nu=\tau.$$
Fix $\gamma,\delta\in P^*$ such that $x=\gamma\delta^{-1}$.  Say $\gamma=m_0tm_1\ldots tm_i$ and $\delta= n_0t n_1\ldots n_j$ in normal form.
  Then
 $$\gamma\delta^{-1}=m_0tm_1\ldots tm_in_j^{-1}t^{-1}\ldots n_1t^{-1}n_0.$$
 Since $\theta(\gamma\delta^{-1}) = \theta(\mu \nu^{-1})=m-n$ we get
 \[i=\theta(\gamma)=\theta(\delta)+m-n\]
 and hence $i\geq m-n$. It follows from the uniqueness of normal form  that there exists $\gamma'\in P^*$ such that
 $$\gamma\delta^{-1}=p_0tp_1t\ldots tp_{m-n-1}t\gamma'\delta^{-1}.$$
 Thus $p_0tp_1t\ldots tp_{m-n-1}t\leq\gamma$. From \eqref{eq-x} we have $(p_0tp_1t\ldots tp_{m-n-1}t)^{-1}x=\sigma\tau^{-1}$. Since $ (p_0tp_1t\ldots tp_{m-n-1}t)^{-1}\gamma,\delta\in P^*$ and $(p_0tp_1t\ldots tp_{m-n-1}t)^{-1}\gamma\delta=\sigma\tau^{-1}$,  we can apply (1) above with $\mu'=\sigma$ and $\nu'=\tau$ to see that $\sigma\leq(p_0tp_1t\ldots tp_{m-n-1}t)^{-1}\gamma$ and $\tau\leq \delta$. Hence $\mu= p_0tp_1t\ldots tp_{m-n-1}t\sigma \leq \gamma$ and $\nu\leq \delta$ as required.

    (3) Suppose that after the cancellations, there are no more $t$ left. Then  $$x=rs^{-1}t^{-1}q_{n-m-1}^{-1}t^{-1}\ldots t^{-1}q_0^{-1}$$   for some $r,s\in P$.   Consider $$x^{-1}=q_0tq_1t\ldots tq_{n-m-1}tsr^{-1}.$$
 By \eqref{Crisp Laca}  there exist $\sigma,\tau\in P$ such that $rs^{-1}=\sigma\tau^{-1}$ and $\sigma\leq r$ and $\tau\leq s$.
Let $$\mu=\sigma \quad\text{and}\quad\nu=q_0tq_1t\ldots tq_{n-m-1}t\tau.$$ By (2) they have the property that $x^{-1}=\nu\mu^{-1}$ and for all $\gamma,\delta\in P^*$ such that $x^{-1}=\delta\gamma^{-1}$ we have $\mu\leq\gamma$ and $\nu\leq\delta$. Taking inverses,   $x=\mu\nu^{-1}$ and for all $\gamma,\delta\in P^*$ such that $x=\gamma\delta^{-1}$ we have $\mu\leq\gamma$ and $\nu\leq\delta$.

    (4) Suppose that  after the initial cancellations there are  both $t$ and $t^{-1}$ left. Then the term with $t$ to the left and $t^{-1}$ to its right is not in $B$. There exist $k\leq m$ and $l\leq n$  such that
    $$x=p_0tp_1t\ldots tp_ktrs^{-1}t^{-1}q_l^{-1}t^{-1}\ldots t^{-1}q_0^{-1}.$$
     By \eqref{Crisp Laca}  there exist $\sigma,\tau\in P$ such that $rs^{-1}=\sigma\tau^{-1}$ and $\sigma\leq r$ and $\tau\leq s$.
   Our candidate for $\mu,\nu$ are $$\mu=p_0tp_1t\ldots tp_kt\sigma\quad\text{and}\quad \nu=q_0tq_1t\ldots tq_lt\tau.$$
Fix  $\gamma,\delta\in P^*$ such that $x=\gamma\delta^{-1}$.  By the argument used in (2), there exists $\gamma'\in P^*$ such that $\gamma\delta^{-1}=p_0tp_1t\ldots tp_{j}t\gamma'\delta^{-1}$. Hence $\gamma'=(p_0tp_1t\ldots tp_{j}t)^{-1}\gamma\in P^*$ and   $p_0tp_1t\ldots tp_{j}t\leq \gamma$.

Consider $\gamma'\delta^{-1}=\sigma\tau^{-1}t^{-1}q_k^{-1}t^{-1}\ldots t^{-1}q_0^{-1}$. Here $\gamma',\delta\in P^*$ and there are no $t$ in $\gamma'\delta^{-1}$ after cancellation.  Applying  (3)
with $\mu'=\sigma$ and $\nu'=\nu$  to get $\mu'=\sigma \leq \gamma'$ and $\nu'=\nu \leq \delta$.
Then   $\mu=p_0tp_1t\ldots tp_{j}t\sigma\leq p_0tp_1t\ldots tp_{j}t\gamma' =\gamma$.
  \end{proof}

Theorem~\ref{(G*,P*) is left ql} gives new examples of quasi-lattice ordered HNN extensions. 

\begin{example}
We can show that the Baumslag-Solitar group $(\BS(c,d),\BS(c,d)^+)$ with $c,d>0$ is quasi-lattice ordered using Theorem~\ref{(G*,P*) is left ql}. Since $(\Z,\N)$ is totally ordered it is quasi-lattice ordered. Let  $A=\{dm:m\in \Z\}$ and $B=\{cm:m\in \Z\}$. Every element $n\in \N$ has a unique decomposition $n=r+md$ where $m\in \N$ and $r\in \{0,1,\ldots, d-1\}$. The remainder $r$ is a choice of coset representative $n+A=r+A$. For all $n'\in n+A\cap \N$ we have $n'=r+md+kd$ where $k\in \Z$ and $k\leq m$. Thus $r\leq n'$.  Hence every coset of $\Z/A$ has nontrivial intersection with $\N$, and has a minimal coset representative in $\N$. Since $B$ is totally ordered it is  closed under taking least upper bounds. Define $\phi:A\to B$ by $\phi(dm)=cm$.  Then $\phi(A\cap\N)=B\cap \N$.  So Theorem~\ref{(G*,P*) is left ql} applies and gives that $(\Z^*,\N^*)$ is quasi-lattice ordered.
\end{example}

\begin{example}\label{example Z^2} We can generalise the previous example  to $(\Z^2,\N^2)$, which is quasi-lattice ordered by \cite[Example~2.3(2)]{Nica1992}. Fix $a,b,c,d\in \N$. Then $A=\{(am,bn):m,n\in \Z\}$ and $B=\{(cm,dn):m,n\in \Z\}$ are subgroups of $\N^2$. Let $\phi:A\to B$ be defined by $\phi((am,bn))=(cm,dn)$. This $\phi$ satisfies $\phi(A\cap \N^2)=B\cap \N^2$.
      For all $(m,n)\in \N^2$, the division algorithm on $\N$ gives a unique decomposition
      \[
      (m,n)=(r_1,r_2)+(ja_1,ka_2) \text{\ for $j,k\in \N$ and $r_1\in \{0,\ldots a_1-1\}$, $r_2\in \{0,\ldots,a_2-1\}$.}
      \] Thus $(r_1,r_2)$ is a minimal left coset representative of $(m,n)+A$. For all $(m,n),(p,q)\in B$, we have
      \[
      (m,n)\vee(p,q)=(\max\{m,p,0\},\max\{n,q,0\})
      \]
       and hence $(m,n)\vee(p,q)\in B$. So $B$ is closed under $\vee$.
      By Theorem~\ref{(G*,P*) is left ql}, $(\Z^{2*},\N^{2*})$ is a quasi-lattice ordered group with presentation $$\Z^{2*}=\langle\Z^2\cup\{t\}\mid (am,bn)t=t(cm,dn)\rangle.$$ It is straightforward to extend this construction to $(\Z^n,\N^n)$.
\end{example}

\begin{example}\label{example free group}
Consider the free group $\F_2$ on $2$ generators $\{a,b\}$ and let $\F_2^+$ be the subsemigroup generated by $e$, $a$ and $b$.  The pair $(\F_2, \F_2^+)$ is quasi-lattice ordered by \cite[Example 2.3(4)]{Nica1992}.  Let $A=\{ a^n:n\in\Z\}$, $B=\{ b^n:n\in\Z\}$ and $\phi:A\to B$ defined by $\phi(a^n)=b^n$.  Every  $x\in F_2^+$  can be written as a product of  $y\in \F_2^+$ which does not end in $a$ followed by $a^n$ for some $n\geq 0$.  Then $y\in xA$.  Every  $z\in yA\cap \F_2^+$ begins with the word $y$ which is in $\F_2^+$. It follows that $y\leq z$. Thus $A$ has minimal left coset representatives in $\F_2^+$. Since $B$ is totally ordered, it is trivially closed under $\vee$. It follows from Theorem~\ref{(G*,P*) is left ql} that $(\F_2^*,\F_2^{+*})$ is a quasi-lattice ordered group with presentation $$\F_2^*=\langle \{a,b,t\}\mid at=tb\rangle.$$
\end{example}

\begin{example}\label{example free group 2} Building on $(\F_2, \F_2^+)$ again, fix $s,u\in \N$, and let $A=\{ a^{ms}:m\in \Z\}$, $B=\{b^{mu}:m\in\Z\}$ and $\phi:A\to B$ be $\phi( a^{ms}) = b^{mu}$. Then $B$ is totally ordered and hence is closed under $\vee$. To see that $A$ has minimal coset representatives, we  observe that every $x\in \F_2^+$ is a product of a  $y\in \F_2^+$ that does not end in $a$ followed by $a^n$ for some $n\in \N$. We   write $n=r+js$ for some $j\in \N$ and $r\in \{0,1,\ldots,s-1\}$. We choose $ya^r$ as our  coset representative. Then for all $z\in ya^rA\cap \F_2^+$  we have $ya^r\leq z$. It follows from Theorem~\ref{(G*,P*) is left ql} that $(\F_2^*,\F_2^{+*})$ is a quasi-lattice ordered group with presentation $$\F_2^*=\langle \{a,b,t\}\mid a^st=tb^u\rangle.$$  Taking $u=s=1$ gives Example~\ref{example free group}.

Replacing $B$ by $B'=\{a^{mu}:m\in\Z\}$ gives a quasi-lattice ordered group $(\F_2^*,\F_2^{+*})$ with presentation $$\F_2^*=\langle \{a,b,t\}\mid a^st=ta^u\rangle.$$
\end{example}

In the next two examples  we show that it is easy to find subgroups which do not have minimal left coset representatives.

\begin{example}
Consider the group   $$G=\Z(\sqrt{2})=\{m+n\sqrt{2}:m,n\in \Z\}$$ with subsemigroup $\Z(\sqrt{2})^+=\Z(\sqrt{2})\cap [0,\infty)$. Let $A=\Z(2\sqrt{2})$. We claim there are no minimal coset representatives for $A$. Suppose, aiming for a contradiction, that there exists some coset representative $p\in \Z(\sqrt{2})^+$ such that
\[
q\in [p+\Z(2\sqrt{2})]\cap \Z(\sqrt{2})^+\Rightarrow p\leq q.
\]
Recall that $\Z(2\sqrt{2})$ is dense in $\R$.\footnote{To see denseness observe that $0<(-2+2\sqrt{2})<1$ and $(-2+2\sqrt{2})^n\in \Z(\sqrt{2})$ for all $n\in \N$. Thus for every open interval $(u,v)$ there exists $n$ such that $(-2+2\sqrt{2})^n<v-u$. Hence there exists $k\in \N$ such that $k(-2+2\sqrt{2})^n\in (u,v)$.} Thus there exists some $a\in \Z(2\sqrt{2})\cap (0,p)$ hence $a\leq p$. Thus $p-a\in [p+\Z(2\sqrt{2})]\cap \Z(\sqrt{2})^+$. But $p-a\leq p$, giving a contradiction.
\end{example}
\begin{example} Consider $(\Z^2,\N^2)$, and let $A$ be the subgroup  generated   by $\{(1,2),(2,1)\}$. Consider the coset \[(1,0)+A=(0,1)+A.\]
Since  $(1,0)$ and $(0,1)$ have no nonzero lower bound,  there can be  no choice of minimal coset representative.
\end{example}

 \section{Amenability of $(G^*, P^*)$}

In this section we  prove the following theorem.

\begin{thm}\label{HNN extensions are amenable as expected}
  Let $(G,P)$ be a quasi-lattice ordered group with subgroups $A$ and $B$. Suppose that  $\phi:A\to B$ is an isomorphism which satisfies the hypotheses of Theorem~\ref{(G*,P*) is left ql}. Let $(G^*,P^*)$ be the corresponding HNN extension. If $(G,P)$ is amenable, then $(G^*,P^*)$ is amenable.
\end{thm}

To prove the theorem we will show that the height map $\theta: G^*\to\Z$ is a controlled map,  that $(\theta^{-1}(e),  \theta^{-1}(e)\cap P^*)$ is amenable, and then apply Theorem~\ref{p-controlled map implies amenable}.  To prove that $(\theta^{-1}(e),  \theta^{-1}(e)\cap P^*)$ is amenable, we start by investigating  order-preserving isomorphisms between the semigroups of quasi-lattice ordered groups.

\begin{lemma}\label{isomorphism preserves least upper bounds}
    Let $(G,P)$ and $(K,Q)$ be quasi-lattice ordered groups. Suppose that there is a semigroup isomorphism $\phi:P\to Q$. Then $\phi$ is  order-preserving. In particular, for $x,y\in P$,  $x\vee y<\infty$ if and only if $\phi(x)\vee\phi(y)<\infty$. If $x\vee y<\infty$ then $\phi(x\vee y)=\phi(x)\vee\phi(y)$.
  \end{lemma}
  \begin{proof}
To see that $\phi$ is order-preserving, let $x,y\in P$ and $x\leq y$. Then $y\in xP$, and $\phi(y)\in \phi(x)\phi(P)=\phi(x)Q$. Thus $\phi(x)\leq\phi(y)$, and $\phi$ is order-preserving.

    To show that $\phi$ preserves the least upper bound structure, first suppose that $x,y\in P$ such that  $x\vee y<\infty$.   Since $\phi$ is order-preserving it follows that $\phi(x),\phi(y)\leq \phi(x\vee y)$. Thus $\phi(x),\phi(y)$ have a common upper bound in $Q$. Hence $\phi(x)\vee\phi(y)$ exists and $\phi(x)\vee\phi(y)\leq \phi(x\vee y)$. Second,  suppose that $\phi(x)\vee\phi(y)<\infty$ for some $x, y\in P$.  Since $\phi^{-1}:Q\to P$ is  a semigroup isomorphism it is order-preserving. Thus $\phi^{-1}(\phi(x)\vee\phi(y))$ is an upper bound for $x=\phi^{-1}(\phi(x))$ and $y=\phi^{-1}(\phi(y))$. Hence $x\vee y$ exists and  $x\vee y\leq \phi^{-1}(\phi(x)\vee\phi(y))$.
     Thus  $$\phi(x\vee y)\leq\phi\circ\phi^{-1}(\phi(x)\vee\phi(y))=\phi(x)\vee\phi(y).$$ Hence $x\vee y<\infty$ if and only if $\phi(x)\vee\phi(y)<\infty$,  and $\phi(x\vee y)=\phi(x)\vee\phi(y)$.
  \end{proof}

  \begin{prop}\label{semigroup isomorphism preserves universal and amenability}
    Let $(G,P)$ and $(K,Q)$ be quasi-lattice ordered groups. Let $\{v_p:p\in P\}$ and $\{w_q:q\in Q\}$ be the generating elements of $C^*(G,P)$ and $C^*(K,Q)$, respectively. Suppose that there is a semigroup isomorphism $\phi:P\to Q$.
        \begin{enumerate}
      \item There exists an isomorphism $\pi_{\phi}:C^*(G,P)\to C^*(K,Q)$ such that $\pi_\phi(v_p)=w_{\phi(p)}$.
      \item $(G,P)$ is amenable if and only if $(K,Q)$ is amenable.
    \end{enumerate}
  \end{prop}

\begin{proof}
(1) We will show that $T:P\to C^*(K,Q)$ defined by $T_p=w_{\phi(p)}$ is a covariant representation of $P$, and then take $\pi_\phi:=\pi_T$. Fix $p,q\in P$. Since $\phi$ is a semigroup isomorphism we have $$T_pT_q=w_{\phi(p)}w_{\phi(q)}=w_{\phi(pq)}=T_{pq}$$ and $T_e=w_{\phi(e_G)}=w_{e_K}=1$. Hence $T$ is an isometric representation.  We have
\begin{align*}
  T_pT^*_pT_qT^*_q&=w_{\phi(p)}w^*_{\phi(p)}w_{\phi(q)}w^*_{\phi(q)}\\
                  &=\begin{cases}w_{\phi(p)\vee\phi(q)}w^*_{\phi(p)\vee\phi(q)}&\text{if $\phi(p)\vee\phi(q)<\infty$}\\0&\text{otherwise}\end{cases}\\
                  \intertext{which, using Lemma~\ref{isomorphism preserves least upper bounds}, is equivalent to}
                  &=\begin{cases}w_{\phi(p\vee q)}w^*_{\phi(p\vee q)}&\text{if $p\vee q<\infty$}\\0&\text{otherwise}\end{cases}\\
                  &=\begin{cases}T_{p\vee q}T^*_{p\vee q}&\text{if $p\vee q<\infty$}\\0&\text{otherwise.}\end{cases}
\end{align*}
Hence $T$ is covariant. By the universal property of $C^*(G,P)$ there exists a homomorphism $\pi_\phi:C^*(G,P)\to C^*(K,Q)$ such that $\pi_\phi(v_p)=T_p=w_{\phi(p)}$.

Since $\phi^{-1}:Q\to P$ is an isomorphism the argument above gives a homomorphism $\pi_{\phi^{-1}}:C^*(K,Q)\to C^*(G,P)$ such that $\pi_{\phi^{-1}}(w_q)=v_{\phi^{-1}(q)}$. In particular,   $$\pi_{\phi^{-1}}(\pi_\phi(v_p))=\pi_{\phi^{-1}}(w_{\phi(p)})=v_{\phi^{-1}(\phi(p))}=v_p$$ and $\pi_{\phi}(\pi_{\phi^{-1}}(w_q))=w_q$ It follows that $\pi_\phi$ is an isomorphism from $C^*(G,P)$ to $C^*(K,Q)$.

(2) By symmetry it suffices to show that if $(K,Q)$ is amenable then $(G,P)$ is amenable. Let $E_Q$ and $E_P$ be the conditional expectations on $C^*(K,Q)$ and $C^*(G,P)$, respectively.  To see $E_P=\pi_{\phi}^{-1}\circ E_Q\circ \pi_{\phi}$ we compute  on spanning elements:
\begin{align*}
\pi_{\phi}^{-1}\circ E_Q\circ \pi_{\phi}(v_pv^*_q)&=\pi_{\phi}^{-1}\circ E_Q(w_{\phi(p)}w^*_{\phi(q)})\\
                                            &=\begin{cases}\pi_{\phi}^{-1}(w_{\phi(p)}w^*_{\phi(q)})&\text{if $\phi(p)=\phi(q)$}\\0&\text{otherwise}\end{cases}\\
                                            &=\begin{cases}\pi_{\phi}^{-1}(w_{\phi(p)}w^*_{\phi(q)})&\text{if $p=q$}\\0&\text{otherwise}\end{cases}\\
                                            &=\begin{cases}v_pv^*_q&\text{if $p=q$}\\0&\text{otherwise}\end{cases}\\
                                            &=E_P(v_pv^*_q).
 \end{align*}
Suppose that $(K,Q)$ is amenable. Then $E_Q$ is faithful. Suppose that $a\geq 0$ and $E_P(a)=0$. Then $\pi_\phi(a)\geq 0$, and
\[
0=E_P(a)=\pi_{\phi}^{-1}\circ E_Q\circ \pi_{\phi}(a) \Rightarrow 0=E_Q\circ \pi_{\phi}(a) \Rightarrow 0=\pi_{\phi}(a)
\]
because $E_Q$ is faithful. Since $\pi_{\phi}$ is faithful, $a=0$. Now $E_P$ is faithful, and hence $(G, P)$ is amenable.
\end{proof}

Next we need some lemmas which will be used to show that the height map $\theta$ is a controlled map. In particular we need to identify the minimal elements of Definition~\ref{expanded controlled map definition}.
If $x\in P^*$ has   normal form $$x=p_0tp_1t\ldots p_{n-1}tp_n$$ we  call $p_0tp_1t\ldots p_{n-1}t$ the \emph{stem} of $x$ and write $$\stem(x)=p_0tp_1t\ldots p_{n-1}t.$$ The set of stems is our candidate for the minimal elements.

\begin{lemma}\label{p,q have cub in P* iff they have cub in P}   Let $(G,P)$ be a quasi-lattice ordered group with subgroups $A$ and $B$. Suppose that  $\phi:A\to B$ is an isomorphism which  satisfies the hypotheses of Theorem~\ref{(G*,P*) is left ql}. Let $p,q\in P$. Then $p$ and $q$ have a common  upper bound in $P^*$ if and only if $p$ and $q$ have a common  upper bound in $P$.
\end{lemma}

\begin{proof}
First suppose that $p$ and $q$ have a common   upper bound $r\in P$. Then $r\in P^*$ and so $r$ is a common  upper bound for $p$ and $q$ in $P^*$.

Second, suppose that $p$ and $q$ have a common upper bound $x\in P^*$. If $\theta(x)=0$, then $x\in P$ and we are done. Suppose, aiming for a contradiction, that $\theta(x)=k$ for some $k\geq 1$, and that $p,q$ have no common upper bound $y$ with $\theta(y)<k$.

  Observe that $p^{-1}x,q^{-1}x\in P^*$, and that $\theta(p^{-1}x)=\theta(x)=\theta(q^{-1}x)=k$. We write $p^{-1}x$ and $q^{-1}x$ in their normal forms:
  $$p^{-1}x=p_0tp_1t\ldots p_{k-1}tp_k\quad\text{and}\quad q^{-1}x=q_0tq_1t\ldots q_{k-1}tq_k,$$
  where $p_i,q_i\in L_A$ for $i<k$ and $p_k,q_k\in P$.
  Consider $$p^{-1}q=p^{-1}x(q^{-1}x)^{-1}=p_0tp_1t\ldots p_{k-1}tp_kq_k^{-1}t^{-1}q_{k-1}^{-1}\ldots t^{-1}q_0^{-1}.$$
  Since $p^{-1}x(q^{-1}x)^{-1}=p^{-1}q$ and $p^{-1}q$ is in normal form we must have some cancellation. Since the first $k$ terms are already in normal form, $p_kq_k^{-1}\in B$. By Lemma~\ref{PP-1-closed is vee closed}, there exist $b_1,b_2\in B\cap P$ such that $b_1\leq p_k$, $b_2\leq q_k$ and $b_1b_2^{-1}=p_kq_k^{-1}$.
Then
  \begin{align*}
   p^{-1}q&= p^{-1}x(q^{-1}x)^{-1}=p_0t\ldots p_{k-1}t(p_kq_k^{-1})t^{-1}q_{k-1}^{-1}\ldots t^{-1}q_0^{-1}\\
    &=p_0t\ldots p_{k-1}t(b_1b_2^{-1})t^{-1}q_{k-1}^{-1}\ldots t^{-1}q_0^{-1} \\
    &=p_0t\ldots p_{k-1}\phi^{-1}(b_1)tt^{-1}\phi^{-1}(b_2)^{-1}q_{k-1}^{-1}\ldots t^{-1}q_0^{-1} \\
    &=p_0t\ldots p_{k-1}\phi^{-1}(b_1)\phi^{-1}(b_2)^{-1}q_{k-1}^{-1}\ldots t^{-1}q_0^{-1}.
  \end{align*}
Rearranging gives $$p(p_0t\ldots tp_{k-1}\phi^{-1}(b_1))=q(q_0tq_1t\ldots tq_{k-1}\phi^{-1}(b_2))\in P^*.$$
Therefore $y=p(p_0tp_1t\ldots p_{k-1}\phi^{-1}(b_1))$ is a common upper bound for $p$ and $q$ in $P^*$ and $\theta(y)=k-1$, giving us the contradiction we sought. Therefore $p$ and $q$ have a common upper bound $y$ with $\theta(y)=0$, and hence they have a common  upper bound in $P$.
  \end{proof}

The statement of Lemma~\ref{stem preserves least upper bounds} is adapted from \cite[Lemma 3.4]{HuefClarkRaeburn2015}.

\begin{lemma}\label{stem preserves least upper bounds}
 Let $x,y\in P^*$ such that $x\vee y<\infty$. Write $$x=\stem(x)p \text{ and } y=\stem(y)q \text{\ where $p,q\in P.$}$$
   \begin{enumerate}
     \item If $\theta(x)=\theta(y)$, then $\stem(x)=\stem(y)$ and $p\vee q<\infty$. In particular, $x\vee y=\stem(x)(p\vee q)$ and $\theta(x\vee y)=\theta(x)=\theta(y)$.
     \item If $\theta(x)<\theta(y)$, then there exists $r\in P$ such that $x\vee y=yr$ and $\theta(x\vee y)=\theta(y)$.
   \end{enumerate}
In particular, $\theta(x \vee y)=\max\{\theta(x), \theta(y)\}$.
\end{lemma}
\begin{proof}
  (1) Suppose that $\theta(x)=\theta(y)$. We know that $x\leq x\vee y$ and $y\leq x\vee y$. Thus, by the uniqueness of normal forms, $\stem(x)=\stem(y)$. Now by left invariance of the partial order we see that 
\[p=\stem(x)^{-1}x\leq\stem(x)^{-1}(x\vee y)\quad\text{and}\quad
q=\stem(x)^{-1}y\leq\stem(x)^{-1}(x\vee y).
\] 
Therefore $p$ and $q$ have a common left upper bound in $P^*$ and hence, by Lemma~\ref{p,q have cub in P* iff they have cub in P}, they have a common left upper bound in $P$ and $p\vee q$ exists in $P$.
  By left invariance $x\vee y=\stem(x)p\vee q$. Further, $\theta(x\vee y)=\theta(x)=\theta(y)$.

  (2) Suppose that $\theta(x)<\theta(y)$. Since $x\leq x\vee y$ we have $x^{-1}(x\vee y)\in P^*$. We can write $x^{-1}(x\vee y)=\tau \gamma u$ for some $u\in P$,  $\tau\in \Sigma_{\theta(y)-\theta(x)}$ and $\gamma\in\Sigma_{\theta(x\vee y)-\theta(y)}$. Then $x\vee y=x\tau\gamma u$.

  Now we have $x\tau\leq x\vee y$ and $\theta(x\tau)=\theta(x)+(\theta(y)-\theta(x))=\theta(y)$. Write $x\tau=\stem(x\tau) w$ for some $w\in P$. Therefore $x\tau\vee y<\infty$ and $\theta(x\tau)=\theta(y)$ so we can apply (1) to see that $\stem(x\tau)=\stem(y)$ and $x\tau\vee y=\stem(y)(q\vee w)$.

  Now $x\vee y\leq\stem(y)(q\vee w)$. Therefore there exists some $r\in P$ such that $x\vee y=\stem(y)qr=yr$. Then $\theta(x\vee y)=\theta(y)$.

  By (1) and (2) we see that $$\theta(x\vee y)=\begin{cases}\theta(x)&\text{if $\theta(x)=\theta(y)$}\\\theta(y)&\text{if $\theta(x)<\theta(y)$.}\end{cases}$$ Thus $\theta(x \vee y)=\max\{\theta(x), \theta(y)\}$.
\end{proof}

\begin{proof}[Proof of Theorem~\ref{HNN extensions are amenable as expected}]
We will use Theorem~\ref{p-controlled map implies amenable}; to do so we need to show that the height map $\theta:(G^*,P^*)\to (\Z, \N)$ is a controlled map in the sense of Definition~\ref{expanded controlled map definition},  and  that $(\theta^{-1}(e),\theta^{-1}(e)\cap P^*)$ is amenable.

  To see that $\theta$ is order-preserving, let $x, y\in P^*$ such that $x\leq y$. Then $x^{-1}y\in P^*$ and $\theta(x^{-1}y)\geq 0$. So
$
  0\leq \theta(x^{-1}y)=-\theta(x)+\theta(y)
$
  and hence  $\theta(x)\leq \theta(y)$. By Lemma~\ref{stem preserves least upper bounds}, if $x\vee y<\infty$, then $\theta(x\vee y)=\max\{\theta(x),\theta(y)\}=\theta(x)\vee\theta(y)$.

  For every $k\in \N$,   $\Sigma_k$ is complete: if $x\in \theta^{-1}(k)\cap P^*$, then $\stem(x)\in \Sigma_k$ and $x=\stem(x)p$ for some $p\in P$. Hence $\stem(x)\leq x$. By the uniqueness of normal forms, if $\sigma,\tau\in \Sigma_k$ and $\sigma\vee \tau<\infty$ then $\sigma=\tau$.
  Therefore $\theta$ is a controlled map into the amenable group $\Z$.

  Suppose that $(G,P)$ is amenable. Then $\theta^{-1}(0)\cap P^*$ is the set of elements of $P^*$ with height $0$, and hence they all have normal form $p_0$ for some $p_0\in P$. Thus $\theta^{-1}(0)\cap P^*$ is isomorphic to $P$. Since $(G,P)$ is amenable, so is $(\theta^{-1}(0),\theta^{-1}(0)\cap P^*)$ by Lemma~\ref{semigroup isomorphism preserves universal and amenability}. Since $(\Z,\N)$ is amenable, it now follows from Theorem~\ref{p-controlled map implies amenable} that $(G^*,P^*)$ is an amenable quasi-lattice ordered group.
\end{proof}

\begin{example}
Since $(\Z^2,\N^2)$ and $(\F_2,\F_2^+)$ are amenable quasi-lattice ordered groups \cite[\S 5.1]{Nica1992}, Theorem~\ref{HNN extensions are amenable as expected} implies that the HNN extensions $(\Z^{2*},\N^{2*})$ and $(\F_2^{*},\F_2^{+*})$ in Examples~\ref{example Z^2}-\ref{example free group 2} are amenable quasi-lattice ordered groups.
\end{example}

  \bibliographystyle{amsplain}

\end{document}